\documentclass
{amsart}

\usepackage{amsmath,amsthm,amsfonts,amssymb,mathtools}
\usepackage{mathrsfs}
\makeatletter
\@namedef{subjclassname@2020}{\textup{2020} Mathematics Subject Classification}
\makeatother
\usepackage
{hyperref}
\usepackage{capt-of}
\usepackage{relsize}

\newcommand{\Be}{\begin{equation}}
\newcommand{\Ee}{\end{equation}}
\newcommand{\Bea}{\begin{eqnarray}}
\newcommand{\Eea}{\end{eqnarray}}
\newcommand{\Bel}{\begin{align}}
\newcommand{\Eel}{\end{align}}
\newcommand{\Beas}{\begin{eqnarray*}}
\newcommand{\Eeas}{\end{eqnarray*}}
\newcommand{\Benu}{\begin{enumerate}}
\newcommand{\Eenu}{\end{enumerate}}
\newcommand{\Bi}{\begin{itemize}}
\newcommand{\Ei}{\end{itemize}}
\newcommand\supp{\operatorname{supp}}

\def\R{{\mathbb R}}
\def\Z{{\mathbb Z}}

\theoremstyle{plain}
\newtheorem{thm}{Theorem}[section]

\newtheorem{lem}[thm]{Lemma}
\newtheorem{prop}[thm]{Proposition}

\theoremstyle{remark}
\newtheorem{rmk}{Remark}  
\theoremstyle{definition}

\numberwithin{equation}{section}

\newcommand{\RNum}[1]{\uppercase\expandafter{\romannumeral #1\relax}}

\newcommand{\bA}{\mathbb A}

\usepackage{wrapfig}
\usepackage{tikz}

\begin{document}

\makeatletter
\@namedef{subjclassname@2020}{\textup{2020} Mathematics Subject Classification}
\makeatother
\subjclass[2020]{Primary 42B25,  Secondary 35S30}
\keywords{Spheircal maximal function, weighted inequality}

\author[Juyoung Lee]{Juyoung Lee}

\address{Research Institute of Mathematics, Seoul National University, Seoul 08826, Republic of Korea}
\email{ljy219@snu.ac.kr}

\title[The critical weighted inequalities of the spherical maximal function]
{The critical weighted inequalities of the spherical maximal function}

\begin{abstract}   
Weighted inequality on the Hardy-Littlewood maximal function is completely understood while it is not well understood for the spherical maximal function. For the power weight $|x|^{\alpha}$, it is known that the spherical maximal operator on $\R^d$ is bounded on $L^p(|x|^{\alpha})$ only if $1-d\leq \alpha<(d-1)(p-1)-d$ and under this condition, it is known to be bounded except $\alpha=1-d$. In this paper, we prove the case of the critical order, $\alpha=1-d$.
\end{abstract}

\maketitle

\section{Introduction}
Let $d\geq 2$ and $d\sigma$ be the normalized Lebesgue measure on $\mathbb{S}^{d-1}$. Then, for $t>0$, we define the spherical average of $f$ by
\[ A_tf(x)=\int_{\mathbb{S}^{d-1}}f(x-ty)d\sigma(y). \]
It is very clear that $A_t$ is a bounded operator on $L^p$ for any $1\leq p\leq \infty$. When we take the supremum on $t>0$, we get the following (global) spherical maximal function:
\[ Mf(x)=\sup_{t>0}|A_tf(x)|. \]
Stein \cite{S2} proved that $M$ is bounded on $L^p$ if and only if $p>d/(d-1)$ when $d\geq 3$. Later, Bourgain \cite{B} obtained that $M$ is bounded on $L^p$ if and only if $p>2$ when $d=2$. A few years later, Mockenhaupt, Seeger, and Sogge \cite{MSS} proved the same result using the local smoothing estimate.

So far, many researches have been devoted to studying the spherical maximal function in various settings. For example, we can define the local spherical maximal function as follows:
\[ M_cf(x)=\sup_{1<t<2} |A_tf(x)|. \]
We have so called $L^p$-improving phenomenon for $M_c$, which means that $M_c$ is bounded from $L^p$ to $L^q$ for some $p<q$. Schlag, Sogge \cite{SchS} obtained the sharp range of $p,q$ except for some boundary points, and S. Lee \cite{L} obtained the result on the boundary lines except the case $(p,q)=(5/2,5)$.

In this paper, we are interested in the weighted inequalities of $M$. This was first considered by Duoandikoetxea, Vega \cite{DV}. They proved that $M$ is bounded on $L^p(|x|^{\alpha})$ if $p>d/(d-1)$ and $1-d<\alpha<(d-1)(p-1)-1$. This result is sharp except the critical case, $\alpha=1-d$. For the critical weight $|x|^{1-d}$, it is known that $M$ is bounded on $L^p_{rad}(|x|^{1-d})$ when $1<p\leq \infty$ for $d\geq 3$, and when $2<p\leq \infty$ for $d=2$ (see \cite{DMO}).  Here, $L^p_{rad}(|x|^{1-d})$ is the subspace of $L^p(|x|^{1-d})$ consists of radial functions. Recently, Nowak, Roncal, Szarek \cite{NRS} obtained sharp conditions for the spherical mean Radon transform on radial functions to be bounded on weighted spaces. Meanwhile, Lacey \cite{La} proved that sparse bounds on the spherical maximal function implies some weighted inequalities. In \cite{La}, the author tried to characterize the class of weight which makes the spherical maximal function bounded, while it is mentioned that $A_p$-weight is not a correct tool to characterize such weight (see also \cite{CGcG}). The bilinear weighted inequalities of the spherical maximal function is also considered by Roncal, Shrivastava, Shuin \cite{RSS}.

The following is the main theorem which says that $M$ is bouned on the critical weighted space $L^p(|x|^{1-d})$.

\begin{thm}\label{main thm}
\begin{enumerate}
    \item Let $d=2$ Then, $M$ is bounded on $L^p(|x|^{-1})$ if and only if $p>2$.
    \item Let $d\geq 3$. Then, $M$ is bounded on $L^p(|x|^{1-d})$ when $p\geq 2$.
\end{enumerate}
\end{thm}
We believe that when $d\geq 3$, $M$ is also bounded on $L^p(|x|^{1-d})$ when $d/(d-1)<p<2$. To prove Theorem \ref{main thm}, we first consider $M_c$ and decompose it using the Littlewood-Paley decomposition. To recover the estimate for $M$, we use a modification of the standard argument (see, for example, \cite{B}, \cite{Sch}). For the purpose, we define the Littlewood-Paley decomposition. Let $\chi$ be a smooth function such that $\supp\chi\subset (1-10^{-2},2+10^{-2})$ and $1=\sum_{j\in\Z}\chi(s/2^j)$ for all $s\neq 0$. We define the projection $\mathcal{P}_j$ by $\widehat{\mathcal{P}_j f}(\xi)=\widehat{f}(\xi)\chi_j(|\xi|)$ where $\chi_j(x)=\chi(s/2^j)$. Then, we define
\[ M_jf(x)=M_c\mathcal{P}_jf(x). \]
Also, we denote $\tilde{\chi} : \R^d\to \R$ by a smooth function which is supported on $\{ x : 1/100<|x|<100 \}$ and equals to 1 when $1/10<|x|<10$, and then $\tilde{\chi}_j(x)=\tilde{\chi}(x/2^j)$. We obtain sharp estimates of $M_j$ as follows.
\begin{prop}\label{mj sph}
    Let $d\geq 2$, $j,k\geq 0$, and $d/(d-1)<p<\infty$. Then, for $\delta=\delta(d,p)>0$, we have
    \begin{align}\label{mj sph est}
        &\Vert \chi_{-k}(|\cdot|)M_jf\Vert_{L^p(|x|^{1-d})}^p\\
        &\lesssim 2^{-\delta|j-k|}\Vert \tilde{\chi}\mathcal{P}_j f\Vert_{L^p(|x|^{1-d})}^p+\sum_{m\in\Z} 2^{-Nj-(d-1)|m|-\delta k} \Vert \chi_m\mathcal{P}_jf \Vert_{L^p(|x|^{1-d})}^p\nonumber
    \end{align}
    for any $N>0$.
\end{prop}
\begin{rmk}
    In the above proposition, the main term is obviously the first term on the right hand side of \eqref{mj sph est}, and the other terms come from the Schwartz tail. Throughout the paper, we will encounter similar issues several times. Since those terms have sufficiently nice decay, they are always harmless. Nevertheless, we rigorously handle all those situations. The reader who believes that the Schwartz tail is harmless may ignore these parts, precisely, every term containing $2^{-Nj}$.
\end{rmk}

We briefly explain the novelty of this paper. Consider a suitable bounded function $F_j$ which is supported in an annulus $\{ x : ||x|-1|<2^{-j} \}$ and the Fourier transform of $F_j$ is supported on $\{ \xi : |\xi|\sim 2^j\}$. This example says that $\Vert M_j F_j\Vert_{L^p(|x|^{1-d})}\gtrsim 2^{-j/p}\approx \Vert F_j\Vert_{L^p(|x|^{1-d})}$ since $|M_jF_j(x)|\gtrsim 1$ when $|x|\lesssim 2^{-j}$. To prove that the spherical maximal function is bounded, it is essential that $\Vert M_j f\Vert_{L^p}\lesssim 2^{-c j}\Vert f\Vert_{L^p}$ for a positive number $c>0$. However, on the weighted space $L^p(|x|^{1-d})$, we cannot expect this exponential decay as we can see at the above. To overcome this difficulty, we decompose $x$ dyadically. Proposition \ref{mj sph} tells us that the main contribution of $M_j$ on $L^p(|x|^{1-d})$ occurs when $|x|\sim 2^{-j}$. This gives us the orthogonality and allows us to prove our main theorem. The proof of Proposition \ref{mj sph} has two parts, $j> k$ and $j\leq k$. In the former case, we use the local smoothing estimate of the wave operator. However, since we are considering the critical weight, $\epsilon$-loss on the order of smoothing is not allowed. When $d\geq 4$, we have the local smoothing estimates without the $\epsilon$-loss (see \cite{HNS}) while we do not have such estimates when $d=2,3$. However, when we consider $L^p$--$L^q$ estimate, we have the critical local smoothing estimate. Thus, we use the $L^p$--$L^q$ local smoothing estimates and H\"older's inequality to recover the $L^p$-estimate. Since $j>k$, the singularity of $|x|^{1-d}$ is harmless. In the latter case $j\leq k$, the blow-up of $|x|^{1-d}$ is no more allowable, so we handle this case using the spherical harmonic expansion. The key idea is that we handled the $t$-integration of the wave operator $\mathcal{W}_t$ using the Plancherel theorem, as a Fourier transform of $|\xi|$, the modulus of the frequency variable.

\subsection*{Notations}
\begin{enumerate}
    \item We denote $A\lesssim B$ when there is an implicit constant $C>0$ which is independent of $A$ and $B$. In the context, $C$ may depend on $p, d, N$. Sometimes, we explicitly denote the implicit constant $C_{a,b}$ which depends on $a,b$.
    \item We abuse several notations for the $L^p$-spaces. $L^p(w)$ denotes the $L^p$-space on a suitable domain with the weight $w$. The domain will depend on the context, but it will be clear. Also, we denote $L^p_{x,t}(\R^d\times I)$ by the $L^p$-space with variable $(x,t)\in\R^d\times I$. Finally, we also denote $L^p_{x}(w; D)$ by the $L^p$-space with the variable $x$, the domain $D$ and the weight $w$.
\end{enumerate}

\subsection*{Acknowledgement}
The author would like to thank Kalachand Shuin for a careful reading of the first version of the draft and finding some typos.

\section{Preliminaries}
\subsection{Basic properties}\label{section basic prop}
We first review some basic properties which are widely used in various literature. For a rectangle $R$ in $\R^d$, we denote $R^\ast$ by the dual rectangle of $\R^d$ centered at the origin. Then, we can find a smooth function $\varphi_R$ such that $|\varphi_R(x)|\sim 1$ on $R$, rapidly decreasing outside of $R$, and $\widehat{\varphi_R}$ is supported on $R^\ast$. For example, when $R$ is a cube centered at the origin with sidelength $r$, then $\varphi_R$ satisfies
\[ |\varphi_R(x)|\lesssim (1+\frac{|x|}{r})^{-N}. \]
Then, we have the local orthogonality on $R$.
\begin{lem}\label{local ortho}
    For a rectangle $R$ and a family of functions $\{ h_i\}$, we assume that $\{ \supp\widehat{h_i}+R^\ast \}$ is a family of finitely overlapping annulus of the form $\{ \xi : |\xi|\sim 2^i \}$. Then we have
    \[ \int_R \sum_{i}|h_i(x)|^pdx\lesssim \int |\sum_{i}h_i(x)\varphi_R(x)|^pdx \]
    when $p\geq 2$.
\end{lem}
This lemma is straightforward from $|\varphi_R(x)|\sim 1$ and the Littlewood-Paley theory. The following lemma tells us that we may consider the averaging operator to control the maximal function.
\begin{lem}[\cite{L}, see also \cite{LL}]\label{sobo}
    Let $1\leq p\leq \infty$ and $F : \R^d\times [1,2]\to \R$ be a smooth function. Then, for any $\lambda>0$ and $x\in\R^d$, we have
    \[ \sup_{t\in[1,2]}|F(x,t)|\lesssim \lambda^{\frac{1}{p}}\Vert F(x,\cdot )\Vert_{L^p([1,2])}+\lambda^{-\frac{1}{p'}}\Vert \partial_t F(x,\cdot)\Vert_{L^p([1,2])}. \]
\end{lem}
Next, we introduce the asymptotic expansion of $\widehat{d\sigma}$ using the asymptotic formula of the Bessel function (see, for example, \cite{S}):
\begin{equation}\label{bessel}
  \textstyle   \widehat{d\sigma}(\xi)=\sum_{\pm, \,0\leq k\leq N}\,C_j^{\pm}|\xi|^{-\frac{d-1}{2}-k}e^{\pm i|\xi|}+E_N(|\xi|),\quad |\xi|\geq 1,
\end{equation}
where $E_N$ is a smooth function satisfying $|({d}/{dr})^l E_N(r)|\leq Cr^{-l-\frac{N+1}{4}},$ $0\leq l\leq (N+1)/{4}$ for $r\geq 1$ and a constant $C>0$. By the Fourier inversion, we have
\[ A_t\mathcal{P}_jf(x)=\int e^{ix\cdot\xi}\chi_j(|\xi|)\widehat{f}(\xi)\widehat{d\sigma}(t\xi)d\xi. \]
By fixing sufficiently large $N>0$, the contribution from $E_N$ is very small. Also, from \eqref{bessel}, it suffices to handle the contribution from $|\xi|^{-(d-1)/2}e^{i|\xi|}$ since the others are similarly, and easily handled.

\subsection{Local smoothing estimates for the wave operator}\label{local smoothing section}
The local smoothing estimate is one of the common tool in the study of maximal function. The wave operator defined by
\[ \mathcal{W}_tf(x)=\int e^{i(x\cdot\xi+t|\xi|)}\widehat{f}(\xi)d\xi \]
is deeply related to the spherical maximal function. As we have seen in Section \ref{section basic prop}, the main contribution of $A_t\mathcal{P}_j$ comes from $\mathcal{W}_t[|D|^{-1/2}\mathcal{P}_jf]$. We denote
\[ {\mathbb A}_\lambda=   \{  \eta\in \R^2:  2^{-1}\lambda \le   |\eta|\le 2\lambda\} \]
for $\lambda>0$. The sharp $L^p$-estimates of $\mathcal{W}$ is obtained by Guth, Wang, and Zhang \cite{GWZ}. By interpolation with the trivial $L^1$-$L^{\infty}$ estimate, we get the following sharp $L^p$-$L^q$ estimates with $\epsilon$-loss.

\begin{prop}[\cite{GWZ}, see also \cite{SchS}, \cite{LL}]\label{prop:locals}
Let  $2\le p\leq q$, $1/{p}+3/{q}\le  1$, and $\lambda \ge 1$. Then,  the estimate   
\begin{equation}\label{locals-est}
\big\Vert \mathcal{W}_t g \big\Vert_{L^q_{x,t}(\R^2\times \mathbb [1,2]])}\le C_{\epsilon,p,q}
\lambda^{( \frac{1}{2}+\frac{1}{p}-\frac{3}{q} )+\epsilon }\Vert g\Vert_{L^p} 
\end{equation}
holds for a constant $C_{\epsilon, p, q}>0$ for any $\epsilon>0$ whenever $\supp \widehat g\subset \bA_\lambda$.  
\end{prop}

However, our problem concerns the weight of the critical order, so the $\epsilon$-loss is not allowable. when $d=2$, S. Lee \cite{L} obtained the critical $L^p$-$L^q$ local smoothing estimate for some $p<q$ using Wolff's sharp bilinear cone restriction estimate (see \cite{W2}). Using the same method, we can obtain the critical $L^p$-$L^q$ local smoothing estimate for some $p<q$ when $d\geq 3$ (without $\epsilon$-loss). However, for our purpose, when $d\geq 3$, estimates from the Strichartz estimate which is also mentioned in \cite{L} is enough (see (1.10) in \cite{L}).

\begin{prop}[\cite{L}]\label{critical pq locals}
    Let $2\leq p\leq q$, $\frac{1}{p}+\frac{d+1}{(d-1)q}\leq 1$, and $\lambda\geq 1$. Also, we assume $q> 14/3$ when $d=2$ and $q\geq \frac{2(d+1)}{d-1}$ when $d\geq 3$. Then, for a constant $C_{p,q,d}>0$, we have the following estimates for any $g$ such that $\supp\widehat{g}\subset \mathbb{A}_{\lambda}$:
    \begin{equation}\label{locals-est critical}
        \Vert \mathcal{W}_tg\Vert_{L^q_{x,t}(\R^d\times [1,2])}\le C_{p,q,d}\lambda^{\frac{d-1}{2}+\frac{1}{p}-\frac{d+1}{q}}\Vert g\Vert_{L^p}.
    \end{equation}
\end{prop}
When $d\geq 4$, Heo, Nazarov, Seeger \cite{HNS} obtained the critical local smoothing estimates of $\mathcal{W}$ for $p>2(d-1)/(d-3)$.
\begin{prop}\label{critical pp locals}
    Let $d\geq 4$ and $p>2(d-1)/(d-3)$. Then, we have the following estimates
    \[ \Vert \mathcal{W}_tg \Vert_{L^p_{x,t}(\R^d\times [1,2])}\leq C_{p,d} \lambda^{\frac{d-1}{2}-\frac{d}{p}}\Vert g\Vert_{L^p} \]
    for some constant $C_{p,d}>0$ whenever $\supp\widehat{g}\subset \mathbb{A}_{\lambda}$.
\end{prop}
We finally remark that allowing the $\epsilon$-loss, the sharp local smoothing estimates beyond the bilinear method is obtained by Beltran and Saari \cite{BS}.

\subsection{Spherical harmonics}
In this section we review some basic properties of spherical harmonics. We refer Chapter 4 of \cite{SW} for details. Let $\mathcal{H}_k$ be the space of spherical harmonics of degree $k$. This means that any element of $\mathcal{H}_k$ is a restriction of a degree $k$ harmonic polynomial on $\mathbb{S}^{d-1}$. Then, it is well known that the collection of all finite linear combinations of elements of $\bigcup_{k=0}^{\infty}\mathcal{H}_k$ is dense in $L^2(\mathbb{S}^{d-1})$. We fix an orthonormal basis of $\mathcal{H}_k$ by
\[ \mathcal{B}_k= \{ Y_i^k : 1\leq i\leq d_k \} \]
where $d_k$ is the dimension of $\mathcal{H}_k$. Then, $\bigcup_{k=0}^{\infty}\mathcal{B}_k$ forms an orthonormal basis of $L^2(\mathbb{S}^{d-1})$. Also, for any function $f\in L^2(\R^d)$, we may find the unique expansion
\begin{equation}\label{sp har expansion}
    f(x)=\sum_{k=0}^{\infty}\sum_{i=1}^{d_k}a_i^k(r)Y_i^k(\theta)
\end{equation}
where $x=r\theta$ for $r>0$, $\theta\in\mathbb{S}^{d-1}$. The following is an interesting property of spherical harmonics which is crucial in our argument.
\begin{lem}[Corollary 2.9, Lemma 2.18, \cite{SW}]\label{sp har lemma}
    Suppose $f(x)=f_0(r)Y(\theta)$ where $x=r\theta$, $r>0$, $\theta\in\mathbb{S}^{d-1}$, and $Y\in\mathcal{H}_k$ for $k\geq 0$. Then, we have
    \[ \widehat{f}(\xi)=\left\{ \int_0^{\infty}f_0(r)\phi_k(r\rho)r^{d-1}dr \right\}Y(\omega) \]
    for a function $\phi_k$ where $\xi=\rho\omega$ for $\rho>0$ and $\omega\in\mathbb{S}^{d-1}$. Also, for a constant $C(d)>0$, we have $|\phi_k(s)|\leq C(d)$ for any $k\geq 0$, $s\in\R$.
\end{lem}
Indeed, Lemma \ref{sp har lemma} is not stated as above in \cite{SW}, but it is contained in the proof. Also, we note that Lemma \ref{sp har lemma} implies
\begin{equation}\label{harmonic piece fourier transform}
    \int_{\mathbb{S}^{d-1}} Y(\theta)e^{-ir\rho \theta\cdot\omega}d\theta=\phi_k(r\rho)Y(\omega)
\end{equation}
for any $Y\in\mathcal{H}_k$. We may find the expansion for $\widehat{f}$ as follows:
\begin{equation}\label{sp har expansion fourier}
    \widehat{f}(\xi)=\sum_{k=0}^{\infty}\sum_{i=1}^{d_k}b_i^k(\rho)Y_i^k(\omega)
\end{equation}
where $\xi=\rho\omega$ for $\rho>0$, $\omega\in\mathbb{S}^{d-1}$. Then, by the orthogonality, we see that
\[ \Vert f\Vert_{L^2}=(\sum_{k=0}^{\infty}\sum_{i=1}^{d_k}\int_0^{\infty} |b_i^k(\rho)|^2\rho^{d-1}d\rho)^{\frac{1}{2}} \]
holds.

\section{Proof of the main theorem}
In this section, we prove Theorem \ref{main thm} assuming Proposition \ref{mj sph}. For $n\in\Z$, we define $\mathcal{P}_{<n}$ by $(\mathcal{P}_{<n}f)^{\wedge}(\xi)=\sum_{j<n}\chi_j(|\xi|)\widehat{f}(\xi)$. Also, we denote $\chi_{<n}(s)=\sum_{j<n}\chi_j(s)$. Before the proof of the main theorem, we introduce a technical lemma handling Schwartz tails.
\begin{lem}\label{2^k maximal lemma}
        Suppose $p>d/(d-1)$. For $k,j\geq 0$ and any $N>0$,
        \begin{align}\label{2^k maximal}
            & \int_{|x|\sim 2^k}|M_jf(x)|^p|x|^{1-d}dx \\
            & \lesssim 2^{-\beta j}\int_{|x|\sim 2^k}|\mathcal{P}_jf(x)|^p|x|^{1-d}dx+\sum_{m\in\Z}2^{-N(j+k)-(d-1)|m|}\int_{|x|\sim 2^m}|\mathcal{P}_jf(x)|^p|x|^{1-d}dx \nonumber
        \end{align}
        holds for some $\beta=\beta(d,p)>0$.
    \end{lem}
Heuristically, since $M_c$ is a local maximal operator, the contribution of $M_jf$ on the annulus $\mathbb{A}_{2^k}$ comes from the values of $\mathcal{P}_jf$ on an $O(1)$-neighborhood of the same annulus. However, the restriction of $\mathcal{P}_jf$ on the annulus is no more localized on the Fourier side. Therefore, we are obliged to handle the Schwartz tail which is the second term on the right hand side of \eqref{2^k maximal}.
    
    \begin{proof}[Proof of lemma \ref{2^k maximal lemma}]
         As above, we obtain that
    \[ \int_{|x|\sim 2^k}|M_c[(1-\tilde{\chi}_k)\mathcal{P}_jf](x)|^p|x|^{1-d}dx=0. \]
    Also, from \cite{B} and \cite{S2} (see also \cite{Sch}), we have
    \begin{equation}\label{spherical no weight}
        \int |M_jg(x)|^pdx\lesssim 2^{-\alpha j}\int |\mathcal{P}_jg(x)|^pdx\lesssim 2^{-\alpha j}\int |g(x)|^pdx
    \end{equation}
    for some $\alpha>0$ depending on $d,p$ when $p>d/(d-1)$. We denote $\mathcal{P}_{<j-10}=\mathcal{Q}_j^1$, and $\sum_{j-10\leq j'\leq j+10}\mathcal{P}_{j'}=\mathcal{Q}_j^2$. Then, by the triangle inequality, we have
    \begin{align}\label{tail decomp}
        &\int_{|x|\sim 2^k}|M_jf(x)|^p|x|^{1-d}dx\\
        &\lesssim 2^{k(1-d)}(\sum_{i=1,2}\Vert M_c[\mathcal{Q}_j^i\tilde{\chi}_k\cdot\mathcal{P}_jf]\Vert_{L^p}+\sum_{j'>j+10}\Vert M_c[\mathcal{P}_{j'}\tilde{\chi}_k\cdot\mathcal{P}_jf] \Vert_{L^p})^p.\nonumber
    \end{align}
    Note that $(\mathcal{Q}_j^1\tilde{\chi}_k\cdot\mathcal{P}_jf)^{\wedge}(\xi)\neq 0$ only if $|\xi|\sim 2^j$ while for $j'>j+10$, $(\mathcal{P}_{j'}\tilde{\chi}_k\cdot\mathcal{P}_j f)^{\wedge}(\xi)\neq 0$ only if $|\xi|\sim 2^{j'}$. Thus, by \eqref{spherical no weight} and the boundedness of $M_c$, we obtain that the left hand side of \eqref{tail decomp} is bounded by the sum of the following three quantities.
    \begin{equation}\label{t1}
        2^{k(1-d)-\alpha' j}\int |\mathcal{Q}_j^1\tilde{\chi}_k\cdot\mathcal{P}_jf(x)|^pdx,
    \end{equation}
    \begin{equation}\label{t2}
        2^{k(1-d)}\int|\mathcal{Q}_j^2\tilde{\chi}_k\cdot\mathcal{P}_jf(x)|^pdx,
    \end{equation}
    \begin{equation}\label{t3}
        2^{k(1-d)-\alpha' j'}\sum_{j'>j+10}\int |\mathcal{P}_{j'}\tilde{\chi}_k\cdot\mathcal{P}_jf(x)|^pdx.
    \end{equation}
    Here, $\alpha'>0$ is a suitable number which is smaller than $\alpha$. From the spirit of the uncertainty principle, $\widehat{\tilde{\chi}_k}$ is essentially supported in a ball of radius $2^{-k}$ centered at the origin. Thus, we may expect that \eqref{t1} will be the main term and the others are error terms. We first see \eqref{t1}. By definition, we have
    \begin{align*}
        \mathcal{Q}_j^1\tilde{\chi}_k(x) & =\int \widehat{\tilde{\chi}_k}(\xi)\chi_{<j-10}(\xi)e^{ix\cdot\xi}d\xi\\
        & =\int \tilde{\chi}_k(y)\big[ \int e^{i(x-y)\cdot\xi}\chi_{<j-10}(\xi)d\xi \big]dy.
    \end{align*}
    This implies that when $|x|\sim 2^m$ for $m\in\Z$,
    \begin{equation}\label{t1 est}
        |\mathcal{Q}_j^1\tilde{\chi}_k(x)|\lesssim \begin{cases}
        2^{-N(j+max\{m.k\})}, & |m-k|\geq 10,\\
        1 & |m-k|<10
    \end{cases}
    \end{equation}
    holds for any $N>0$. Thus, \eqref{t1} is bounded by the right hand side of \eqref{2^k maximal} as desired.

    Next, we see \eqref{t2}. Similarly as before, we obtain that when $|x|\sim 2^m$ for $m\in\Z$,
    \begin{equation}\label{t2 est}
        |\mathcal{Q}_j^2\tilde{\chi}_k(x)|\lesssim \begin{cases}
        2^{-N(j+max\{m.k\})}, & |m-k|\geq 10,\\
        2^{-N(j+k)} & |m-k|<10
    \end{cases}
    \end{equation}
    holds for any $N>0$. Again, this implies \eqref{t2} is bounded by the right hand side of \eqref{2^k maximal}, as desired.

    Finally, we see \eqref{t3}. The proof is almost similar as before. By the same argument, we obtain that when $|x|\sim 2^m$ for $m\in\Z$,
    \begin{equation}\label{t3 est}
        |\mathcal{P}_{j'}\tilde{\chi}_k(x)|\lesssim \begin{cases}
        2^{-N(j'+\max\{ m,k\})}, & |m-k|\geq 10,\\
        2^{-N(j'+k)}, & |m-k|<10
    \end{cases}
    \end{equation}
    holds for any $N>0$. Since $j'>j+10$, by taking summation in $j'$, we obtain that \eqref{t3} is bounded by the right hand side of \eqref{2^k maximal}, which completes the proof.
    \end{proof}

\begin{proof}[Proof of Theorem \ref{main thm}]
    By a scaling argument, it suffices to consider
    \[ \tilde{M}f(x)=\sup_{0<t<1}|A_tf(x)|. \]
    For each $n\geq 0$, we can easily check that
    \[ \sup_{t\sim 2^{-n}}|A_t\mathcal{P}_{<n}f(x)|\lesssim M_{HL}f(x) \]
    holds independent of $n$, where $M_{HL}$ is the Hardy-Littlewood maximal operator on $\R^d$. Therefore, we have
    \begin{align}\label{first global decomp}
        & \int |\tilde{M}f(x)|^p|x|^{1-d}dx\\
        & \lesssim \int \big(\sup_{n\geq 0}\sup_{t\sim 2^{-n}}|A_t\mathcal{P}_{<n}f(x)|^p+\sup_{n\geq 0}\sup_{t\sim 2^{-n}}|\sum_{j\geq n}A_t\mathcal{P}_jf(x)|^p\big)|x|^{1-d}dx\nonumber \\ 
        & \lesssim \int |M_{HL}f(x)|^p|x|^{1-d}dx +\sum_{k\in\Z,n\geq 0}\int_{|x|\sim 2^{-k}}\sup_{t\sim 2^{-n}}|\sum_{j\geq n}A_t\mathcal{P}_jf(x)|^p|x|^{1-d}dx.\nonumber
    \end{align}
    It is well known that
    \[ \int |M_{HL}f(x)|^p|x|^{1-d}dx\lesssim \int |f(x)|^p|x|^{1-d}dx \]
    since $|x|^{1-d}$ is an $A_p$-weight. Thus, we only consider the second part of the right handed side of \eqref{first global decomp}.

    For simplicity, we denote $f_n(x)=f(x/2^n)$ for a function $f$. Notice that $[\mathcal{P}_jf]_n=\mathcal{P}_{j-n}f_n$. Then, we have
    \begin{align*}
        \int_{|x|\sim 2^{-k}}\sup_{t\sim 2^{-n}}|\sum_{j\geq n}A_t\mathcal{P}_jf(x)|^p|x|^{1-d}dx =2^{-n}\int_{|x|\sim 2^{-k+n}}|\sum_{j\geq n}M_{j-n}f_n(x)|^p|x|^{1-d}dx.
    \end{align*}
    Now there are two cases, $k<n$ and $k\geq n$. We first handle the case $k<n$. In this case, we have
    \begin{align}\label{k<n case}
        & \sum_{n\geq 0}\sum_{k<n}2^{-n}\int_{|x|\sim 2^{-k+n}}|\sum_{j\geq n}M_{j-n}f_n(x)|^p|x|^{1-d}dx\\
        & =\sum_{n\geq 0}2^{-n}\sum_{k>0}\int_{|x|\sim 2^k}|\sum_{j\geq 0}M_jf_n(x)|^p|x|^{1-d}dx\nonumber \\ 
        & \leq \sum_{n\geq 0}2^{-n}\sum_{k>0}\big( \sum_{j\geq 0}(\int_{|x|\sim 2^k}|M_jf_n(x)|^p|x|^{1-d}dx)^{\frac{1}{p}} \big)^p \nonumber
    \end{align}
    by the triangle inequality. By Lemma \ref{2^k maximal lemma}, $[\mathcal{P}_jf]_n=\mathcal{P}_{j-n}f_n$, and H\"oler's inequality, we have
    \begin{align*}
        \eqref{k<n case} \lesssim & \sum_{n,k,j\geq 0}2^{-\beta' j}\int_{|x|\sim 2^{k-n}}|\mathcal{P}_{j+n}f(x)|^p|x|^{1-d}dx\\
        & +\sum_{n,k,j\geq 0,\,m\in\Z}2^{-N(j+k)-(d-1)|m|}\int_{|x|\sim 2^{m-n}}|\mathcal{P}_{j+n}f(x)|^p|x|^{1-d}dx\\
        \eqqcolon & I_1+ I_2
    \end{align*}
    where $\beta'>0$ is a suitable number smaller than $\beta$ in Lemma \ref{2^k maximal lemma}.
    
    By changing summations, we get
   \[ I_1= \sum_{j\geq 0, k\in\Z}2^{-\beta' j}\int_{|x|\sim 2^k}\sum_{n\geq \max\{ 0,-k\}}|\mathcal{P}_{j+n}f(x)|^p|x|^{1-d}dx. \]
   By covering $\{ x : |x|\sim 2^k\}$ by cubes with sidelength $\sim 2^k$ which are finitely overlapping and contained in a bigger annulus $\{ x : |x|\sim 2^k\}$, we obtain
   \[ \int_{|x|\sim 2^k}\sum_{n\geq \max\{ 0,-k\}}|\mathcal{P}_{j+n}f(x)|^pdx\lesssim \int |\mathcal{P}_{>\max\{ 0,-k\}}f(x)|^p(1+2^{-k}|x|)^{-N}dx \]
   by Lemma \ref{local ortho}, for a sufficiently large $N>0$. We claim that
   \begin{equation}\label{freq larger est}
       \int |\mathcal{P}_{>-k}f(x)|^p(1+2^{-k}|x|)^{-N}dx\lesssim \int |f(x)|^p(1+2^{-k}|x|)^{-N}dx
   \end{equation}
   for any $k\in\Z$. Indeed, since $\mathcal{P}_{<-k}f=f\ast \psi_{k}$ where $|\psi_k(x)|\lesssim 2^{-dk}(1+2^{-k}|x|)^{-N}$ for any $N>0$, \eqref{freq larger est} comes from the following simple inequality:
   \[ 2^{-dk}\int (1+2^{-k}|x-y|)^{-N}(1+2^{-k}|x|)^{-N}dx\lesssim (1+2^{-k}|y|)^{-N}. \]
   It suffices to check this for $k=0$ by scaling, and it is straightforward. Therefore,
   \begin{align}\label{I_1 final}
       I_1 & \lesssim \sum_{j\geq 0}2^{-\beta' j}\int (|\mathcal{P}_{>0}f(x)|^p+|f(x)|^p)\sum_{k\in\Z}2^{k(1-d)}(1+2^{-k}|x|)^{-N}dx\\
       & \lesssim \int (|\mathcal{P}_{>0}f(x)|^p+|f(x)|^p)|x|^{1-d}dx \nonumber \\
       & \lesssim \int |f(x)|^p|x|^{1-d}dx\nonumber
   \end{align}
   as desired, since $|x|^{1-d}$ is an $A_p$-weight. For the later use, we remark the following inequality:
   \begin{equation}\label{singi weight}
       \sum_{n\in\Z}2^{n(d-1)}(1+2^n|x|)^{-N}\lesssim |x|^{1-d}.
   \end{equation}

   Similarly, we now control $I_2$. Since
   \[\sum_{n,k,j,m\geq 0}2^{-N(j+k)-(d-1)|m|}\int_{|x|\sim 2^{m-n}}|\mathcal{P}_{j+n}f(x)|^p|x|^{1-d}dx\lesssim I_1, \]
   we have that
   \[ I_2\lesssim I_1+\sum_{n,j\geq 0}2^{-Nj+n(d-1)}\int_{|x|\lesssim 2^{-n}}|\mathcal{P}_{j+n}f(x)|^pdx \]
   holds. As before, we use the local orthogonality, Lemma \ref{local ortho} with parameter $j$. Then, we obtain
   \[ \sum_{j\geq 0}\int_{|x|\lesssim 2^{-n}}|\mathcal{P}_{j+n}f(x)|^pdx\lesssim \int |\mathcal{P}_{>n}f(x)|^p(1+2^n|x|)^{-N}dx \]
   for any $N>0$. By \eqref{singi weight}, we finally obtain
   \[ I_2\lesssim \int |f(x)|^p|x|^{1-d}dx \]
   as desired, by \eqref{freq larger est}.

    Now we consider the other case, $k\geq n$. In this case, we have
    \begin{align}\label{k>n eq1}
        & \sum_{n\geq 0}\sum_{k\geq n}2^{-n}\int_{|x|\sim 2^{-k+n}}|\sum_{j\geq n}M_{j-n}f_n(x)|^p|x|^{1-d}dx\\ \nonumber
        & =\sum_{n\geq 0}2^{-n}\sum_{k\geq 0}\int_{|x|\sim 2^{-k}}|\sum_{j\geq 0}M_jf_n(x)|^p|x|^{1-d}dx.
    \end{align}
    By H\"older's inequality,
    \begin{align}\label{key holder}
        |\sum_{j\geq 0}M_jf_n(x)|^p & \leq (\sum_{j\geq 0}|M_jf_n(x)|^p2^{\delta_0|j-k|})(\sum_{j\geq 0}2^{-\delta_0|j-k|/(p-1)})^{p-1}\\
        & \lesssim \sum_{j\geq 0}|M_jf_n(x)|^p2^{\delta_0|j-k|}\nonumber
    \end{align}
    holds for $\delta_0>0$. Putting \eqref{key holder} into \eqref{k>n eq1}, we obtain
    \[ \eqref{k>n eq1}\lesssim \sum_{n\geq 0}2^{-n}\sum_{k,j\geq 0}2^{\delta_0|j-k|}\int_{|x|\sim 2^{-k}} |M_jf_n(x)|^p|x|^{1-d}dx.  \]
    By Proposition \ref{mj sph}, we obtain
    \begin{align}\label{applying prop}
        &\int_{|x|\sim 2^{-k}}|M_jf_n(x)|^p|x|^{1-d}dx\\
        &\lesssim 2^{-\delta|j-k|}\int_{|x|\sim 1} |\mathcal{P}_jf_n(x)|^p|x|^{1-d}dx+ \sum_{m\in\Z}2^{-Nj-(d-1)|m|-\delta k} \int_{|x|\sim 2^{m}}|\mathcal{P}_jf_n(x)|^p|x|^{1-d}dx \nonumber\\
        & \eqqcolon I_3+I_4\nonumber
    \end{align}
    when $p=2$ for $d\geq 3$, and when $p>2$ for $d=2$. We first see the contribution from $I_3$. We choose $\delta_0=\delta/2$ so that we get exponential decay. From the relation $[\mathcal{P}_{j+n}f]_n=\mathcal{P}_jf_n$, we have
    \begin{align*}
        \sum_{n\geq 0}2^{-n}\sum_{k,j\geq 0}2^{\delta_0|j-k|}I_3 & \lesssim \sum_{n,k,j\geq 0}2^{-\delta|j-k|/2}\int_{|x|\sim 2^{-n}}|\mathcal{P}_{j+n}f(x)|^p|x|^{1-d}dx\\
        & \lesssim \sum_{n,j\geq 0}\int_{|x|\sim 2^{-n}}|\mathcal{P}_{j+n}f(x)|^p|x|^{1-d}dx.
    \end{align*}
    For each fixed $n\geq 0$, we use the local orthogonality. As above, by covering $\{ x : |x|\sim 2^{-n} \}$ by cubes with sidelength $\sim 2^{-n}$ which are finitely overlapping and contained in a bigger annulus $\{ x : |x|\sim 2^{-n} \}$, we have
    \begin{align*}
        \sum_{j\geq 0}\int_{|x|\sim 2^{-n}}|\mathcal{P}_{j+n}f(x)|^p|x|^{1-d}dx & \lesssim 2^{n(d-1)}\int|\mathcal{P}_{>n}f(x)|^p(1+2^{n}|x|)^{-N}dx\\
        & \lesssim 2^{n(d-1)}\int |f(x)|^p(1+2^{n}|x|)^{-N}dx.
    \end{align*}
    By \eqref{singi weight}, the contribution from $I_3$ is as desired.

    Now we see the contribution from $I_4$. Choosing sufficiently large $N$, $\delta_0=\delta/2$, and the relation $[\mathcal{P}_{j+n}f]_n=\mathcal{P}_jf_n$, we get
    \begin{align*}
        &\sum_{n\geq 0}2^{-n}\sum_{k,j\geq 0}2^{\delta_0|j-k|}I_4\\
        & \lesssim \sum_{n,k,j\geq 0,\, m\in\Z}2^{-Nj-(d-1)|m|-\delta k/2}\int_{|x|\sim 2^{m-n}}|\mathcal{P}_{j+n}f(x)|^p|x|^{1-d}dx.
    \end{align*}
    Splitting the summation in $m\in\Z$ by $m\geq 0$ and $m<0$, this is bounded by the sum of the following two quantities:
    \[ I_5=\sum_{n,j\geq 0}2^{-Nj+n(d-1)}\int_{|x|\lesssim 2^{-n}}|\mathcal{P}_{j+n}f(x)|^pdx, \]
    \[ I_6=\sum_{n,j,m\geq 0} 2^{-Nj-m(d-1)}\int_{|x|\sim 2^{m-n}}|\mathcal{P}_{j+n}f(x)|^p|x|^{1-d}dx. \]
    Indeed, these are already bounded by $I_2$ and thus, we obtain
    \[ \sum_{n\geq 0}2^{-n}\sum_{k,j\geq 0}2^{\delta_0|j-k|}I_4\lesssim \int |f(x)|^p|x|^{1-d}dx \]
    and this concludes the proof.    
\end{proof}

\section{Proof of Proposition \ref{mj sph}}
The proof of Proposition \ref{mj sph} is different depending on whether $j\leq k$ or $j> k$. In the former case, we use Lemma \ref{sp har lemma}, the property of spherical harmonics. In the latter case, we use the argument in \cite{DV}.
\begin{proof}[Proof of Proposition \ref{mj sph} when $j\leq k$]
    We first claim the following inequality for some $\delta(d,p)>0$ when $d/(d-1)<p< \infty$:
    \begin{equation}\label{jk piece without local}
        \Vert \chi_{-k}(|\cdot|)M_jf\Vert_{L^p(|x|^{1-d})}\lesssim 2^{-\delta(k-j)}\Vert \mathcal{P}_jf\Vert_{L^p}.
    \end{equation}
    Note that the estimate does not depend on whether $k$ is bigger than $j$ or not, and we have the $L^p$-norm without the weight on the right hand side. We prove \eqref{jk piece without local} with $\delta=0$ when $p=d/(d-1)$, and $\delta(d,2)>0$ when $p=2$. Also, we have a trivial estimate with $\delta=0$ when $p\infty$. Then, \eqref{jk piece without local} follows by interpolation. When $p=d/(d-1)$, the estimate is already proved in \cite{DV}. Indeed, inequality (3) in the proof of Theorem 10 in \cite{DV} has $2^{j\epsilon}$-loss. However, authors provided sharp estimates without $\epsilon$-loss restricting $x$ in a small annulus with width $2^{-k}$. Thus, we only focus on $p=2$. Let
    \begin{equation}\label{sp har expansion in proof}
        \widehat{f}(\xi)=\sum_{k=0}^{\infty}\sum_{i=1}^{d_k}b_i^k(\rho)Y_i^k(\omega)
    \end{equation}
    for $\xi=\rho\omega$. By \eqref{bessel} and Lemma \ref{sobo}, it suffices to prove the following:
    \begin{equation}\label{goal of j<k}
        \Vert \chi_{-k}(|x|)\int e^{i(x\cdot\xi+t|\xi|)}\chi_j(|\xi|)\widehat{f}(\xi)d\xi\Vert_{L^2_{x,t}(|x|^{1-d};\R^d\times[1,2])}\lesssim 2^{\frac{d-2}{2}j-\delta(k-j)}\Vert \mathcal{P}_jf\Vert_{L^2}.
    \end{equation}
    Putting \eqref{sp har expansion in proof} into the square of the left hand side of \eqref{goal of j<k}, we get
    \begin{align*}
        &\int \Big\vert \chi_{-k}(r)\int e^{i(r\rho \theta\cdot\omega+t\rho)}\chi_j(\rho)\sum_{k,i}b_i^k(\rho)Y_i^k(\omega)\rho^{d-1}d\rho d\omega \Big\vert^2 drd\theta dt\\
        =&\sum_{k,i}\int \Big\vert \chi_{-k}(r)\int e^{it\rho}\chi_j(\rho)b_i^k(\rho)\phi_k(-r\rho)\rho^{d-1}d\rho \Big\vert^2 drdt
    \end{align*}
    by \eqref{harmonic piece fourier transform} and the orthogonality. For each $k,i$, Plancherel's theorem implies
    \begin{align*}
        &\int \Big\vert \chi_{-k}(r)\int e^{it\rho}\chi_j(\rho)b_i^k(\rho)\phi_k(-r\rho)\rho^{d-1}d\rho \Big\vert^2 drdt\\
        \leq &\int \Big\vert \chi_{-k}(r)\chi_j(\rho)b_i^k(\rho)\phi_k(-r\rho)\rho^{d-1} \Big\vert^2 drd\rho\\
        \approx &\int |\chi_j(\rho)b_i^k(\rho)|^2\rho^{d-1}\Big\{ \int |\phi_k(-r)\chi(2^{k-j}r)|^2\rho^{d-2} dr \Big\}d\rho.
    \end{align*}
    By Lemma \ref{sp har lemma}, we obtain
    \[ \int |\phi_k(-r)\chi(2^{k-j}r)|^2\rho^{d-2}dr\lesssim 2^{(d-2)j+j-k} \]
    when $\rho\sim 2^j$, and the implicit constants are independent of $k,i$. Thus, we can deduce \eqref{goal of j<k} when $p=2$, and thus, completes the proof of \eqref{jk piece without local}.

    Now we need to use localization arguments as in the proof of Lemma \ref{2^k maximal lemma} to reach the goal. Since $M_c$ is a local maximal operator, we obtain
    \begin{equation}\label{local cutoff est}
        \int_{|x|\sim 2^{-k}}|M_c [(1-\tilde{\chi})\mathcal{P}_jf](x)|^p|x|^{1-d}dx=0.
    \end{equation}
    As before, by the triangle inequality,
    \begin{align}\label{localization triangle}
        &\int_{|x|\sim 2^{-k}}|M_jf(x)|^p|x|^{1-d}dx\\
        &\lesssim (\sum_{i=1,2}\Vert \chi_{-k}M_c[\mathcal{Q}_j^i\tilde{\chi}\cdot\mathcal{P}_jf] \Vert_{L^p(|x|^{1-d})} +\sum_{j'>j+10}\Vert \chi_{-k}M_c[\mathcal{P}_{j'}\tilde{\chi}\cdot\mathcal{P}_jf]  \Vert_{L^p(|x|^{1-d})} )^p.\nonumber
    \end{align}
    Note that $[\mathcal{Q}_j^i\tilde{\chi}\cdot\mathcal{P}_jf]^{\wedge}$ is supported on $\{ \xi : |\xi|\sim 2^j  \}$, and $[\mathcal{P}_{j'}\tilde{\chi}\cdot\mathcal{P}_jf]^{\wedge}$ is supported on $\{ \xi : |\xi|\sim 2^{j'}  \}$ when $j'>j+10$. Thus, we can apply \eqref{jk piece without local}. As we have seen already, the main term will come from $i=1$, and the other terms will be error terms. In the proof of Lemma \ref{2^k maximal lemma}, we obtained estimates for $\mathcal{Q}_j^i\tilde{\chi}_k$ and $\mathcal{P}_{j'}\tilde{\chi}_k$, \eqref{t1 est} and \eqref{t3 est}. We use these estimates for $k=0$. The calculation is almost the same with that of the proof of Lemma \ref{2^k maximal lemma}. Thus, the proposition is concluded if we can handle the term with $i=2$.
    
    When $i=2$, we have
    \begin{align*}
        &\Vert \chi_{-k}M_c[\mathcal{Q}_j^2\tilde{\chi}\cdot\mathcal{P}_jf]\Vert_{L^p(|x|^{1-d})} \\
        &\leq \Vert \chi_{-k}M_c\mathcal{P}_{<0}[\mathcal{Q}_j^2\tilde{\chi}\cdot\mathcal{P}_jf] \Vert_{L^p(|x|^{1-d})} + \sum_{0\leq m\leq j}\Vert \chi_{-k}M_c\mathcal{P}_m[\mathcal{Q}_j^2\tilde{\chi}\cdot\mathcal{P}_jf]\Vert_{L^p(|x|^{1-d})}.
    \end{align*}
    We momentarily denote $g(x)=\mathcal{Q}_j^2\tilde{\chi}\cdot\mathcal{P}_jf(x)$. The first term on the right hand side is bounded by
    \[ 2^{\frac{k(d-1)}{p}}\Big(\int_{|x|\sim 2^{-k}} \Big|\int \frac{|g(x-z)|}{(1+|z|)^N}dz\Big|^p dx\Big)^{\frac{1}{p}}\lesssim 2^{-\frac{k}{p}}\Big(  \int \frac{|g(x)|^p}{(1+|x|)^{N/2}}dx \Big)^{\frac{1}{p}} \]
    for any $N>0$. Also, the second term is bounded by
    \[ \sum_{0\leq m<j} 2^{-\delta(k-m)}\Vert g\Vert_{L^p}\lesssim 2^{-\delta(k-j)}\Vert g\Vert_{L^p}. \]
    Therefore, by the estimate \eqref{t2 est}, the proof is concluded.
\end{proof}

\begin{proof}[Proof of Proposition \ref{mj sph} when $j>k$]
    Indeed, it is essentially proved in \cite{DV}. But we provide an alternative proof. We first claim the following inequality for some $\delta(d,p)>0$ when $d/(d-1)<p<\infty$:
    \begin{equation}\label{j>k piece without local}
        \Vert \chi_{-k}(|\cdot|)M_jf\Vert_{L^p(|x|^{1-d})}\lesssim 2^{-\delta(j-k)}\Vert \mathcal{P}_j f\Vert_{L^p}.
    \end{equation}
    Again, \eqref{j>k piece without local} with $\delta=0$ when $p=d/(d-1),\infty$ is already known. Thus, we choose a suitable $d/(d-1)<p<\infty$ and prove \eqref{j>k piece without local} with $\delta(d,p)>0$. By \eqref{bessel} and Lemma \ref{sobo}, we may replace the left hand side of \eqref{j>k piece without local} by
    \begin{equation}\label{wave operator change}
        2^{j(-\frac{(d-1)}{2}+\frac{1}{p})+\frac{(d-1)k}{p}}\Vert \chi_{-k}(|x|)\mathcal{W}\mathcal{P}_j f \Vert_{L^p_{x,t}(\R^d\times [1,2])}.
    \end{equation}
    We use the critical local smoothing estimate Proposition \ref{critical pp locals} when $d\geq 4$. Precisely, we have
    \[ \Vert \mathcal{W}\mathcal{P}_j f\Vert_{L^p_{x,t}(\R^d\times [1,2])}\lesssim 2^{j(\frac{d-1}{2}-\frac{d}{p})}\Vert \mathcal{P}_jf\Vert_{L^p} \]
    when $p>2(d-1)/(d-3)$. This implies \eqref{wave operator change} is bounded by $2^{-(d-1)(j-k)/p}\Vert \mathcal{P}_jf\Vert_{L^p}$. Therefore, \eqref{j>k piece without local} for the case $d\geq 4$ is proved.

    When $d=2,3$, it is not known that the critical local smoothing estimate holds for $p\neq 2,\infty$. Instead, we use the critical $L^p$-$L^q$ local smoothing estimate, Proposition \ref{critical pq locals}.
    
    Indeed, we have
    \[ \Vert \mathcal{W}\mathcal{P}_j f\Vert_{L^q_{x,t}(\R^d\times [1,2])}\lesssim 2^{j(\frac{d-1}{2}+\frac{1}{p}-\frac{d+1}{q})} \Vert \mathcal{P}_j f\Vert_{L^p} \]
    when $\frac{1}{p}+\frac{d+1}{(d-1)q}\leq 1$, $q>14/3$ for $d=2$ and $q>4$ for $d=3$. This implies
    \[ \Vert \chi_{-k}(|\cdot|) M_j f\Vert_{L^q}\lesssim 2^{j(\frac{1}{p}-\frac{d}{q})}\Vert \mathcal{P}_jf\Vert_{L^p}. \]
    By H\"older's inequality, we have
    \begin{align*}
        \Vert \chi_{-k}(|\cdot|) M_jf\Vert_{L^p(|x|^{1-d})} & \lesssim \Vert \chi_{-k}(|\cdot|)M_jf\Vert_{L^q_{x,t}(|x|^{-d+\frac{q}{p}})}\\
        & \lesssim 2^{-(j-k)(\frac{d}{q}-\frac{1}{p})}\Vert \mathcal{P}_jf\Vert_{L^p}.
    \end{align*}
    Therefore, we get \eqref{j>k piece without local} when $d/q>1/p$. By Proposition \ref{critical pq locals}, we can choose such $p,q$.

    Now we use the localization argument to prove \eqref{mj sph est}. We recall \eqref{local cutoff est} and \eqref{localization triangle}. As before, we use \eqref{j>k piece without local} to $\mathcal{Q}_j^1\tilde{\chi}\cdot\mathcal{P}_j f$ and $\mathcal{P}_{j'}\tilde{\chi}\cdot\mathcal{P}_jf$ when $j'>j+10$. Using \eqref{t1 est} and \eqref{t3 est}, the proof is completely the same. When $i=2$, the proof is simpler than the case of $j\geq k$. We just use the boundedess of $M_c$ and obtain
    \[ \Vert \chi_{-k}M_c[\mathcal{Q}_j^2\tilde{\chi}\cdot\mathcal{P}_jf]\Vert_{L^p(|x|^{1-d})}\lesssim 2^{\frac{k(d-1)}{p}}\Vert \mathcal{Q}_j^2\tilde{\chi}\cdot\mathcal{P}_j f\Vert_{L^p(1;B_{100})} \]
    where $B_{100}$ is the ball centered at the origin and radius $100$. The proof is concluded by \eqref{t2 est} since $j>k$ by choosing a sufficiently large $N>0$.
\end{proof}


\begin{thebibliography}{00}

\bibitem{BS} D. Beltran, O Saari, \emph{$L^p$--$L^q$
 local smoothing estimates for the wave equation via $k$-broad Fourier restriction}, J. Fourier Anal. Appl. \textbf{28} (2022), no. 5, Paper No. 76, 29 pp.

\bibitem{B} J. Bourgain, \emph{Averages in the plane over convex curves and maximal operators}, J. Anal. Math. {\bf 47} (1986), 69--85.


\bibitem{CGcG} M. Cowling, J. Garc\'ia-Cuerva, H. Gunawan, \emph{Weighted estimates for fractional maximal functions related to spherical means}, Bull. Austral. Math. Soc. \textbf{66} (2002), no. 1, 75–90.


\bibitem{DMO} J. Duoandikoetxea, M. Adela, O. Oruetxebarria, \emph{The spherical maximal operator on radial functions}, J. Math. Anal. Appl. \textbf{387} (2012), no.2, 655–666.

\bibitem{DV} J. Duoandikoetxea, L. Vega, \emph{Spherical means and weighted inequalities}, J. London Math. Soc. (2) \textbf{53} (1996), no.2, 343–353.


\bibitem{GWZ} L. Guth, H. Wang, R. Zhang, \emph{A sharp square function estimate for the cone in $\mathbb R^3$},  Ann. of Math.  \textbf{192}  (2020), 551--581.

\bibitem{HNS} Y. Heo, F. Nazarov, and A. Seeger, \emph{Radial Fourier multipliers in high dimensions}, Acta Math. {\bf 206}, (2011), no. 1,  55–92.

\bibitem{La} M. Lacey, \emph{Sparse bounds for spherical maximal functions}, J. Anal. Math. \textbf{139} (2019), no. 2, 613–635.

\bibitem{LL} J. Lee, S. Lee, \emph{$L^p$ maximal bound and Sobolev regularity of two-parameter  averages over tori}, arXiv:2210.13377. 


\bibitem{L} S. Lee, \emph{Endpoint estimates for the circular maximal function}, Proc. Amer. Math. Soc. {\bf 131} (2003), 1433--1442.


\bibitem{MSS} G. Mockenhaupt, A. Seeger, C. Sogge, \emph{Wave front sets, local smoothing and Bourgain's circular maximal theorem}, Ann. of Math. (2) \textbf{136} (1992), no. 1, 207–218.

\bibitem{NRS} A. Nowak, L. Roncal, T. Szarek, \emph{Endpoint estimates and optimality for the generalized spherical maximal operator on radial functions}, Commun. Pure Appl. Anal. \textbf{22} (2023), no. 7, 2233–2277.

\bibitem{RSS} L. Roncal, S. Shrivastava, K. Shuin, \emph{Bilinear spherical maximal functions of product type}, J. Fourier Anal. Appl. \textbf{27} (2021), no. 4, Paper No. 73, 42 pp.

\bibitem{Sch} W. Schalg, \emph{$L^p$--$L^q$ estimates for the circular maximal function}, Ph.D. Thesis, California Institute of Technology (1996).

\bibitem{SchS} W. Schlag, C. Sogge, \emph{Local smoothing estimates related to the circular maximal theorem}, Math. Res. Lett. {\bf 4} (1997), no. 1, 1–15.

\bibitem{S2} E. Stein, \emph{Maximal functions: spherical means}, Proc. Nat. Acad. Sci. USA \textbf{73} (1976), 2174--2175.

\bibitem{S} \bysame, \emph{Harmonic Analysis: Real Variable Methods, Orthogonality and Oscillatory Integrals}, Princeton Univ. Press, Princeton, NJ, 1993.

\bibitem{SW} E. Stein, G. Weiss, \emph{Introduction to Fourier Analysis on Euclidean Spaces}, Princeton Univ. Press, Princeton, NJ, 1975.


\bibitem{W2} \bysame, \emph{A sharp bilinear cone restriction estimate}, Ann. of Math. (2) \textbf{153} no. 3, (2001), 661--698.

\end{thebibliography}
\end{document}